\documentclass[a4paper,twoside,reqno]{amsart}

\usepackage[pagewise]{lineno}
\usepackage[utf8]{inputenc}

\usepackage{authblk,orcidlink}
\usepackage{hyperref}
\hypersetup{urlcolor=blue, citecolor=red}
\usepackage{amsmath,amsthm,amssymb,xcolor,bbm,mathrsfs,graphicx,float, subcaption}
\usepackage[english]{babel}
\usepackage{fancyhdr}
\subjclass[2020]{60J25, 91A22, 60K35}
\keywords{Garbage-offloading process; Fermi function;
evolutionary games on graphs; payoff polarization;
garbage concentration; stochastic dynamics}

\title{Payoff Polarization and Garbage Concentration in a Stochastic Garbage Disposal Game}
\author{Hsin-Lun Li}

\date{}
\email{hsinlunl@math.nsysu.edu.tw}

\theoremstyle{definition}
\newtheorem{theorem}{Theorem}

\newtheorem{lemma}[theorem]{Lemma}
\newtheorem{corollary}[theorem]{Corollary}

\begin{document}\sloppy

\thispagestyle{firstpage}
\maketitle
\begin{center}
    Hsin-Lun Li \orcidlink{0000-0003-1497-1599}
    \centerline{$^1$Department of Applied Mathematics, National Sun Yat-sen University, Kaohsiung 804, Taiwan}
\end{center}
\medskip

\begin{abstract}
    We study a stochastic garbage disposal game on a finite graph in which agents transfer garbage through pairwise interactions governed by payoff-dependent Fermi probabilities. We investigate the long-term behavior of both payoffs and garbage distributions. We prove that payoff dispersion grows at least linearly in time almost surely, leading to payoff polarization, with the maximum and minimum payoffs diverging to positive and negative infinity, respectively. Moreover, the payoff range grows at least on the order of the square root of time. For garbage dynamics, we establish almost-sure concentration on a single agent in the two-agent case for any positive disposal efficiency. When disposal is complete and the underlying graph is connected, we further prove that all garbage becomes concentrated on a single agent after a finite random time, although the agent holding the garbage may change over time. These results characterize the emergence of payoff inequality and garbage concentration in stochastic competitive interactions on networks.
\end{abstract}

\section{Introduction}

The long-term behavior of competitive systems on networks is influenced by both the interaction structure and the mechanisms governing the outcomes of individual interactions. In stochastic models, local interactions can generate substantial differences among agents even when the total quantity of a conserved resource remains constant. Such phenomena arise in evolutionary games, models of social hierarchies, and stochastic exchange models of wealth and income. Evolutionary games on graphs provide a framework for studying how network structure affects competitive dynamics \cite{szabo2007evolutionary,nowak2006evolutionary,perc2010coevolutionary}. Meanwhile, kinetic exchange models demonstrate how repeated pairwise transactions can generate nontrivial distributions of resources and persistent inequality \cite{dragulescu2000statistical,chakraborti2000statistical,chatterjee2004pareto,chatterjee2007kinetic}. These studies motivate the investigation of stochastic interaction systems in which resource redistribution and competitive outcomes are dynamically coupled.

The garbage disposal game provides a framework for studying resource transfer through pairwise competition. In the multilayer setting, the interaction structure is represented by multiple graph layers, allowing the effects of different interaction networks to be incorporated into the dynamics \cite{li2025multilayer}. The present work considers a continuous-time stochastic garbage disposal game on a finite graph. Each edge is equipped with an independent Poisson clock, and an interaction occurs whenever the corresponding clock rings. The two agents connected by the activated edge compete, with the probability of winning determined by their current payoffs through a Fermi function. The winner disposes of a fraction of their garbage and transfers the disposed amount to the loser, while their payoffs change by equal and opposite increments. Consequently, the total amount of garbage and the sum of all payoffs are conserved, although the individual garbage holdings and payoffs evolve over time.

Payoff-dependent pairwise comparison rules are widely used in evolutionary game theory. In particular, the Fermi function provides a probabilistic mechanism through which payoff differences influence competitive outcomes \cite{traulsen2006stochastic}. This mechanism connects naturally to models of self-organizing hierarchies, in which reinforcement of winning and losing experiences can produce pronounced differences in competitive status \cite{bonabeau1995phase,bonabeau1996mathematical}. Related reinforcement mechanisms have also been studied in probability theory \cite{pemantle1988random}. However, in the garbage disposal game considered here, the payoff process and the garbage distribution are coupled: an agent's current payoff influences its probability of winning future interactions, while the outcome of each interaction changes both the payoff configuration and the distribution of garbage. Understanding the resulting long-term behavior therefore requires analyzing these two evolving quantities together.

Our first objective is to characterize the asymptotic behavior of the payoffs. Let \(P_i(t)\) denote the payoff of agent \(i\) at time \(t\), and define the payoff dispersion by
\[
V_t=\sum_{i=1}^{n}P_i(t)^2.
\]
Since the sum of the payoffs is conserved and initially equals zero, \(V_t\) measures the dispersion of payoffs around their common mean. We prove that, almost surely, the payoff dispersion grows at least linearly in time, with an explicit lower bound determined by the payoff increment and the number of edges. This result implies payoff polarization: the maximum payoff tends to positive infinity, whereas the minimum payoff tends to negative infinity. We further establish an explicit lower bound on the asymptotic growth rate of the payoff range. These conclusions show that the conservation of total payoff does not prevent increasingly pronounced payoff inequality.

Our second objective is to investigate the long-term distribution of garbage. We first establish that, for a connected graph with two agents, all garbage asymptotically concentrates on one of the two agents for every disposal efficiency \(\alpha\in(0,1]\). We then consider complete disposal, corresponding to \(\alpha=1\), on an arbitrary finite connected graph. In this case, we prove that there exists an almost surely finite random time after which all garbage is concentrated on a single agent at every subsequent time. The identity of this agent need not remain fixed: the theorem concerns concentration at each time, rather than permanent ownership by one particular agent. These results identify conditions under which stochastic competition leads to complete garbage concentration.

The mathematical analysis relies on the continuous-time jump-process structure of the model and on martingale methods. General foundations for Markov processes and interacting stochastic systems can be found in \cite{ethier2009markov,liggett1985interacting,liggett2013stochastic,harris1974contact}. The martingale convergence arguments are related to standard results in probability theory \cite{durrett2019probability,hall1980martingale}. Although the present model shares structural features with stochastic exchange and interacting-particle systems, its central feature is the feedback between payoff-dependent competition and garbage redistribution.

The main contributions of this paper can be summarized as follows:
\begin{itemize}
    \item We establish an almost-sure linear lower bound on payoff dispersion and derive the resulting payoff polarization and a quantitative lower bound on the growth of the payoff range.
    \item We prove almost-sure garbage concentration for the two-agent model for every positive disposal efficiency.
    \item We establish finite-time garbage concentration for complete disposal on every finite connected graph, without requiring the identity of the agent holding all the garbage to remain fixed.
\end{itemize}

The remainder of the paper is organized as follows. Section~2 introduces the stochastic model and its assumptions. Section~3 presents the main results. The auxiliary lemmas and their proofs are developed in the subsequent sections, followed by numerical simulations illustrating the theoretical findings.

\section{Model}

Let

$$
[n]=\{1,\ldots,n\},
\qquad n\geq 2,
$$
be a finite population of agents. Each agent $i\in[n]$ has a
nonnegative amount of garbage $x_i(t)$ and a payoff
$P_i(t)\in\mathbb{R}$ at time $t\in\mathbb{R}_+$. The garbage state is
represented by

$$
X(t)=\bigl(x_1(t),\ldots,x_n(t)\bigr)^\top
\in\mathbb{R}_+^n.
$$
For each $t\geq 0$, let
\[
x_{(1)}(t)\leq x_{(2)}(t)\leq\cdots\leq x_{(n)}(t)
\]
denote the ordered garbage loads at time $t$.
Let

$$
G=([n],E)
$$
be an undirected social network, where $E$ is the set of socially
connected pairs. For each edge $\{i,j\}\in E$, interactions occur
according to an independent Poisson process with rate $1$. Thus, all
socially connected pairs interact at the same rate. Equivalently, the
total interaction rate is $|E|$, and, conditional on an interaction
occurring, each edge is selected with probability $1/|E|$.

Assume that the initial payoffs satisfy

$$
\sum_{i=1}^nP_i(0)=0.
$$
Let

$$
M=\sum_{i=1}^n x_i(0)
$$
denote the total amount of garbage in the population.

Suppose that the edge $\{i,j\}\in E$ interacts at time $t$. The winner is
determined probabilistically according to the current payoffs. By the
Fermi rule, agent $i$ defeats agent $j$ with probability

$$
p_{ij}(t)
=
\frac{1}
{1+\exp\{-\beta(P_i(t)-P_j(t))\}},
\qquad \beta>0.
$$
Thus, an agent with a higher payoff has a larger probability of winning.
The parameter $\beta$ measures the sensitivity of the competition outcome
to payoff differences.

Let $\alpha\in(0,1]$ denote the disposal efficiency and let $a>0$
denote the payoff reward for winning and the payoff cost incurred by
losing. If $i$ defeats $j$, then $i$ offloads a proportion $\alpha$ of
its current garbage onto $j$ and receives a payoff increase of $a$,
while $j$ incurs a payoff decrease of $a$. Accordingly,

$$
x_i(t+)=(1-\alpha)x_i(t),
\qquad
P_i(t+)=P_i(t)+a,
$$
and

$$
x_j(t+)=x_j(t)+\alpha x_i(t),
\qquad
P_j(t+)=P_j(t)-a.
$$
All other agents retain their garbage and payoffs:

$$
x_k(t+)=x_k(t),
\qquad
P_k(t+)=P_k(t),
\qquad
k\notin\{i,j\}.
$$
If $j$ defeats $i$, the corresponding updates are obtained by
interchanging $i$ and $j$.

The interaction rule preserves both total garbage and total payoff.
Consequently, for every $t\geq0$,

$$
\sum_{i=1}^n x_i(t)
=
\sum_{i=1}^n x_i(0)
=
M
$$
and

$$
\sum_{i=1}^n P_i(t)
=
\sum_{i=1}^n P_i(0)
=
0.
$$
In particular, the garbage state remains in the simplex

$$
\mathcal{S}_M
=
\left\{
x\in\mathbb{R}_+^n:
\sum_{i=1}^n x_i=M
\right\}.
$$

The initial payoffs may depend on the initial garbage distribution. For
example,

$$
P_i(0)=\frac{M}{n}-x_i(0),
\qquad i\in[n],
$$
implies that agents with less garbage have a higher probability of
winning at the initial time. In contrast,

$$
P_i(0)=x_i(0)-\frac{M}{n},
\qquad i\in[n],
$$
implies that agents with more garbage have a higher probability of
winning at the initial time. Both specifications satisfy

$$
\sum_{i=1}^nP_i(0)=0.
$$
The model therefore incorporates the feedback mechanism

$$
\begin{aligned}
\text{garbage level}
&\longrightarrow
\text{payoff}
\longrightarrow
\text{winning probability}\\
&\longrightarrow
\text{garbage disposal}
\longrightarrow
\text{future garbage level}.
\end{aligned}
$$
In particular, winning reduces an agent's own garbage, whereas losing
increases its garbage burden through the garbage received from the
winner.

To quantify the dispersion of payoffs, define

$$
V_t=\sum_{i=1}^nP_i(t)^2.
$$
Since the total payoff is conserved and equals zero,

$$
\frac{1}{n}\sum_{i=1}^nP_i(t)=0
$$
for every $t\geq0$. Hence, $V_t$ measures the squared Euclidean
dispersion of the payoff vector around its population mean. We refer to
$V_t$ as the \emph{payoff dispersion}.

\section{Main Results}\label{sec:main}

This section presents the main asymptotic results of the model. We first establish the almost-sure linear growth of payoff dispersion, which implies unbounded payoff polarization and provides a lower bound on the growth rate of the payoff range. We then investigate the long-term distribution of garbage. For two agents, we show that garbage eventually concentrates on one agent for any disposal efficiency. For a general connected graph, we establish finite-time garbage concentration under complete disposal.

\begin{theorem}[Almost-sure linear growth of payoff dispersion]\label{Thm:Almost-sure linear growth of payoff dispersion}
Under the model assumptions,

$$
\liminf_{t\to\infty}\frac{V_t}{t}
\geq
2a^2|E|
\qquad\text{a.s.}
$$
In particular,

$$
V_t\longrightarrow\infty
\quad\text{as}\quad t\to\infty\quad\text{a.s.}
$$

\end{theorem}

\begin{corollary}[Almost-sure payoff polarization]\label{Corollary:Almost-sure payoff polarization}
Under the model assumptions,

$$
\max_{i\in[n]}P_i(t)\to+\infty
\qquad\text{and}\qquad
\min_{i\in[n]}P_i(t)\to-\infty
\qquad\text{a.s.}
$$

\end{corollary}

\begin{corollary}[Almost-sure growth of the payoff range]\label{Corollary:Almost-sure growth of the payoff range}
Let

$$
R_t=\max_{i\in[n]}P_i(t)-\min_{i\in[n]}P_i(t).
$$
Then

$$
\liminf_{t\to\infty}\frac{R_t}{\sqrt{t}}
\geq
2a\sqrt{\frac{|E|}{n-1}}
\qquad\text{a.s.}
$$
In particular,

$$
R_t\to\infty
\quad\text{as}\quad t\to\infty\quad\text{a.s.}
$$

\end{corollary}

\begin{corollary}[Two-agent garbage concentration]
\label{Corollary:Two-agent-garbage-concentration}
Suppose that $n=2$ and $G$ is connected. Then, for every
$\alpha\in(0,1]$, there exists a random variable $I\in[2]$ such that
\[
x_I(t)\longrightarrow 0,
\qquad
x_{3-I}(t)\longrightarrow M
\qquad\text{as}\quad t\to\infty\quad\text{a.s.}
\]
\end{corollary}

\begin{theorem}[Garbage concentration under complete disposal]
\label{Theorem:Garbage-concentration-alpha-one}
Suppose that $\alpha=1$ and $G$ is connected. Then there exists a finite
random time $T<\infty$ such that
\[
x_{(n)}(t)=M
\qquad\text{for all }t\ge T
\quad\text{a.s.}
\]
Equivalently, all garbage is concentrated on a single agent at each time after some finite random time.
\end{theorem}

\section{Proofs of the Main Results}

This section develops the technical tools needed to establish the main results. We first analyze the evolution of payoff dispersion using the infinitesimal generator and martingale methods. We then investigate the garbage redistribution process on connected graphs, with particular attention to the case of complete disposal.

\subsection{Payoff Dispersion}

We begin by examining the conditional expected change in payoff dispersion following an interaction. This leads to a lower bound on its infinitesimal generator and a corresponding martingale decomposition. We then establish a quadratic variation bound and apply a martingale strong law to obtain the almost-sure growth of payoff dispersion.

\begin{lemma}[Conditional expected increase of payoff dispersion]
Suppose that the edge $\{i,j\}\in E$ interacts at time $t$. Then

$$
\begin{aligned}
&\mathbb{E}\!\left[
\Delta V_t
\,\middle|\,
P_i(t),P_j(t),\{i,j\}\text{ interacts at time }t
\right]\\
&\qquad =
2a^2
+
2a\bigl(P_i(t)-P_j(t)\bigr)
\tanh\left(
\frac{\beta\bigl(P_i(t)-P_j(t)\bigr)}{2}
\right)
\geq 2a^2,
\end{aligned}
$$
where

$$
\Delta V_t=V_{t+}-V_t.
$$

\end{lemma}

\begin{proof}
Let

$$
d_{ij}(t)=P_i(t)-P_j(t).
$$
If $i$ defeats $j$, then

$$
\begin{aligned}
\Delta V_t
&=(P_i(t)+a)^2-P_i(t)^2
 +(P_j(t)-a)^2-P_j(t)^2\\
&=2a^2+2a d_{ij}(t).
\end{aligned}
$$
If $j$ defeats $i$, then

$$
\Delta V_t=2a^2-2a d_{ij}(t).
$$
Therefore,
\begin{align*}
&\mathbb{E}\!\left[
\Delta V_t
\,\middle|\,
P_i(t),P_j(t),\{i,j\}\text{ interacts at time }t
\right]\\
&\quad=
p_{ij}(t)\bigl(2a^2+2a d_{ij}(t)\bigr)
+p_{ji}(t)\bigl(2a^2-2a d_{ij}(t)\bigr)\\
&\quad=
2a^2+2a d_{ij}(t)
\bigl(p_{ij}(t)-p_{ji}(t)\bigr).
\end{align*}
By the Fermi rule,

$$
p_{ij}(t)
=
\frac{1}{1+\exp\{-\beta d_{ij}(t)\}},
$$
and, since

$$
p_{ji}(t)=1-p_{ij}(t),
$$
we obtain

$$
p_{ij}(t)-p_{ji}(t)
=
\tanh\left(\frac{\beta d_{ij}(t)}{2}\right).
$$
Consequently,

$$
\mathbb{E}\!\left[
\Delta V_t
\,\middle|\,
P_i(t),P_j(t),\{i,j\}\text{ interacts at time }t
\right]
=
2a^2
+
2a d_{ij}(t)
\tanh\left(\frac{\beta d_{ij}(t)}{2}\right).
$$
Since $\beta\geq0$,

$$
d_{ij}(t)
\tanh\left(\frac{\beta d_{ij}(t)}{2}\right)\geq0.
$$
Hence

$$
\mathbb{E}\!\left[
\Delta V_t
\,\middle|\,
P_i(t),P_j(t),\{i,j\}\text{ interacts at time }t
\right]
\geq2a^2.
$$

\end{proof}

\begin{lemma}[Generator and martingale decomposition]\label{Lemma:Generator and martingale decomposition}
The generator of $V_t$ is given by

$$
\begin{aligned}
\mathcal{L}V_t
&=
\sum_{\{i,j\}\in E}
\mathbb{E}\!\left[
\Delta V_t
\,\middle|\,
P_i(t),P_j(t),\{i,j\}\text{ interacts at time }t
\right]\\
&=
\sum_{\{i,j\}\in E}
\left[
2a^2
+
2a\bigl(P_i(t)-P_j(t)\bigr)
\tanh\left(
\frac{\beta\bigl(P_i(t)-P_j(t)\bigr)}{2}
\right)
\right]\\
&\geq 2a^2|E|.
\end{aligned}
$$
Consequently, there exists a càdlàg martingale $(M_t)_{t\geq0}$ with $M_0=0$ such that

$$
V_t
=
V_0+\int_0^t\mathcal{L}V_s\,ds+M_t,
\qquad t\geq0.
$$

\end{lemma}

\begin{proof}
Since each edge $\{i,j\}\in E$ has an independent Poisson interaction process with rate $1$, the generator is obtained by summing the conditional expected changes in $V_t$ over all edges. By the previous lemma,

$$
\begin{aligned}
\mathcal{L}V_t
&=
\sum_{\{i,j\}\in E}
\mathbb{E}\!\left[
\Delta V_t
\,\middle|\,
P_i(t),P_j(t),\{i,j\}\text{ interacts at time }t
\right]\\
&=
\sum_{\{i,j\}\in E}
\left[
2a^2
+
2a\bigl(P_i(t)-P_j(t)\bigr)
\tanh\left(
\frac{\beta\bigl(P_i(t)-P_j(t)\bigr)}{2}
\right)
\right]\\
&\geq2a^2|E|.
\end{aligned}
$$
The martingale decomposition follows from the standard Dynkin martingale for the pure-jump Markov process. In particular,

$$
M_t
=
V_t-V_0-\int_0^t\mathcal{L}V_s\,ds
$$
is a càdlàg martingale with $M_0=0$. Therefore,

$$
V_t
=
V_0+\int_0^t\mathcal{L}V_s\,ds+M_t.
$$

\end{proof}

\begin{lemma}[Quadratic variation bound]\label{Lemma:Quadratic variation bound}
The predictable quadratic variation of the martingale $(M_t)_{t\geq0}$ is defined by

$$
\langle M\rangle_t
=
\int_0^t
\sum_{\{i,j\}\in E}
\operatorname{Var}\!\left(
\Delta V_s
\,\middle|\,
P_i(s),P_j(s),\{i,j\}\text{ interacts at time }s
\right)ds.
$$
Then,

$$
\langle M\rangle_t\le Ct,
\qquad t\geq0,
$$
for some constant $C=C(a,\beta,|E|)$. Consequently,

$$
\langle M\rangle_t=O(t).
$$

\end{lemma}

\begin{proof}
For an interaction between agents $i$ and $j$, let

$$
d_{ij}(t)=P_i(t)-P_j(t).
$$
The two possible values of $\Delta V_t$ are

$$
2a^2+2ad_{ij}(t)
\quad\text{and}\quad
2a^2-2ad_{ij}(t),
$$
with probabilities $p_{ij}(t)$ and $p_{ji}(t)$, respectively. Hence

$$
\begin{aligned}
&\operatorname{Var}\!\left(
\Delta V_t
\,\middle|\,
P_i(t),P_j(t),\{i,j\}\text{ interacts at time }t
\right)\\
&\qquad
=16a^2d_{ij}(t)^2p_{ij}(t)p_{ji}(t).
\end{aligned}
$$
By the Fermi rule,

$$
p_{ij}(t)p_{ji}(t)
=
\frac{1}{4}
\operatorname{sech}^2
\left(\frac{\beta d_{ij}(t)}{2}\right).
$$
Therefore,
$$
\operatorname{Var}\!\left(
\Delta V_t
\,\middle|\,
P_i(t),P_j(t),\{i,j\}\text{ interacts at time }t
\right)
=
4a^2d_{ij}(t)^2
\operatorname{sech}^2
\left(\frac{\beta d_{ij}(t)}{2}\right).
$$
Since $\beta>0$, the function

$$
u\longmapsto
u^2\operatorname{sech}^2\left(\frac{\beta u}{2}\right)
$$
is bounded on $\mathbb R$. Thus there exists $C_0=C_0(a,\beta)<\infty$ such that

$$
\operatorname{Var}\!\left(
\Delta V_t
\,\middle|\,
P_i(t),P_j(t),\{i,j\}\text{ interacts at time }t
\right)
\le C_0.
$$
Since each edge has interaction rate $1$,

$$
\langle M\rangle_t
\le C_0|E|t.
$$
Taking $C=C_0|E|$ completes the proof.
\end{proof}

\begin{lemma}[Martingale strong law]\label{Lemma:Martingale strong law}
$(M_t)_{t\geq0}$ is square-integrable and

$$
\frac{M_t}{t}\longrightarrow0
\quad\text{as}\quad t\to\infty\quad\text{a.s.}
$$

\end{lemma}

\begin{proof}
By Lemma~\ref{Lemma:Quadratic variation bound},

$$
\langle M\rangle_t\le Ct,
\qquad t\geq0.
$$
Therefore,

$$
\mathbb{E}\!\left[\langle M\rangle_t\right]
\le Ct<\infty.
$$
Hence, $(M_t)_{t\geq0}$ is square-integrable and

$$
\mathbb{E}[M_t^2]
=
\mathbb{E}\!\left[\langle M\rangle_t\right]
\le Ct.
$$
Moreover,

$$
\frac{\langle M\rangle_t}{t^2}
\leq
\frac{C}{t}
\longrightarrow0
\qquad\text{as }t\to\infty.
$$
The strong law for square-integrable martingales therefore yields

$$
\frac{M_t}{t}\longrightarrow0
\quad\text{as}\quad t\to\infty\quad\text{a.s.}
$$

\end{proof}

\begin{proof}[\bf Proof of Theorem~\ref{Thm:Almost-sure linear growth of payoff dispersion}]
By Lemma~\ref{Lemma:Generator and martingale decomposition},

$$
V_t
=
V_0+\int_0^t\mathcal L V_s\,ds+M_t
$$
and

$$
\mathcal L V_s\geq2a^2|E|.
$$
Hence,

$$
V_t
\geq
V_0+2a^2|E|t+M_t.
$$
Dividing by $t$ gives

$$
\frac{V_t}{t}
\geq
\frac{V_0}{t}
+
2a^2|E|
+
\frac{M_t}{t}.
$$
Since
\[
\frac{V_0}{t}\to0
\quad\text{as}\quad t\to\infty,
\]
and, by Lemma~\ref{Lemma:Martingale strong law},
\[
\frac{M_t}{t}\to0
\quad\text{as}\quad t\to\infty\quad\text{a.s.},
\]
we obtain

$$
\liminf_{t\to\infty}\frac{V_t}{t}
\geq
2a^2|E|
\qquad\text{a.s.}
$$
In particular,

$$
V_t\to\infty
\quad\text{as}\quad t\to\infty\quad\text{a.s.}
$$

\end{proof}

\begin{proof}[\bf Proof of Corollary~\ref{Corollary:Almost-sure payoff polarization}]
Since

$$
V_t=\sum_{i=1}^nP_i(t)^2\to\infty
\quad\text{as}\quad t\to\infty\quad\text{a.s.},
$$
at least one of the positive or negative parts of the payoff vector must become unbounded. Because

$$
\sum_{i=1}^nP_i(t)=0,
$$
the positive and negative parts must balance each other. Therefore,
\[
\max_{i\in[n]}P_i(t)\to+\infty
\quad\text{as}\quad t\to\infty\quad\text{a.s.},
\]
and
\[
\min_{i\in[n]}P_i(t)\to-\infty
\quad\text{as}\quad t\to\infty\quad\text{a.s.}
\]
\end{proof}

\begin{proof}[\bf Proof of Corollary~\ref{Corollary:Almost-sure growth of the payoff range}]
Since

$$
\sum_{i=1}^nP_i(t)=0,
$$
we have

$$
\left(\sum_{i=1}^nP_i(t)\right)^2=0.
$$
Expanding the left-hand side gives

$$
\sum_{i=1}^nP_i(t)^2
+
2\sum_{1\leq i<j\leq n}P_i(t)P_j(t)=0.
$$
Therefore,

$$
2\sum_{1\leq i<j\leq n}P_i(t)P_j(t)
=
-\sum_{i=1}^nP_i(t)^2
=
-V_t.
$$
On the other hand,

$$
\begin{aligned}
\sum_{1\leq i<j\leq n}
\bigl(P_i(t)-P_j(t)\bigr)^2
&=
\sum_{1\leq i<j\leq n}
\left(P_i(t)^2+P_j(t)^2-2P_i(t)P_j(t)\right)\\
&=
(n-1)\sum_{i=1}^nP_i(t)^2
-
2\sum_{1\leq i<j\leq n}P_i(t)P_j(t).
\end{aligned}
$$
Each term $P_i(t)^2$ appears exactly $n-1$ times in the sum. Using

$$
\sum_{i=1}^nP_i(t)^2=V_t
$$
and

$$
2\sum_{1\leq i<j\leq n}P_i(t)P_j(t)=-V_t,
$$
we obtain

$$
\sum_{1\leq i<j\leq n}
\bigl(P_i(t)-P_j(t)\bigr)^2
=
(n-1)V_t+V_t
=
nV_t.
$$
Thus,

$$
V_t
=
\frac{1}{n}
\sum_{1\leq i<j\leq n}
\bigl(P_i(t)-P_j(t)\bigr)^2.
$$
Now, by definition of $R_t$,

$$
\left|P_i(t)-P_j(t)\right|\leq R_t
$$
for every $1\leq i<j\leq n$. Hence,

$$
\begin{aligned}
V_t
&\leq
\frac{1}{n}\binom{n}{2}R_t^2\\
&=
\frac{n-1}{2}R_t^2.
\end{aligned}
$$
Consequently,

$$
\frac{R_t}{\sqrt{t}}
\geq
\sqrt{\frac{2V_t}{(n-1)t}}.
$$
By Theorem~\ref{Thm:Almost-sure linear growth of payoff dispersion},

$$
\liminf_{t\to\infty}\frac{V_t}{t}
\geq
2a^2|E|
\qquad\text{a.s.}
$$
Therefore,

$$
\liminf_{t\to\infty}\frac{R_t}{\sqrt{t}}
\geq
2a\sqrt{\frac{|E|}{n-1}}
\qquad\text{a.s.}
$$
\end{proof}

\begin{proof}[\bf Proof of Corollary~\ref{Corollary:Two-agent-garbage-concentration}]
Since $n=2$ and $G$ is connected, $G$ consists of a single edge, and
hence $|E|=1$. Define
\[
D(t)=P_1(t)-P_2(t).
\]
Then
\[
R_t
=\max_{i\in[2]}P_i(t)-\min_{i\in[2]}P_i(t)
=|P_1(t)-P_2(t)|
=|D(t)|.
\]
By Corollary~\ref{Corollary:Almost-sure growth of the payoff range},
\[
\liminf_{t\to\infty}\frac{|D(t)|}{\sqrt{t}}
\geq 2a
\qquad\text{a.s.}
\]
Consequently,
\[
|D(t)|\longrightarrow\infty
\quad\text{as}\quad t\to\infty\quad\text{a.s.}
\]
Since $D(t)$ changes by exactly $2a$ or $-2a$ at every interaction,
its sign is therefore eventually constant. Thus, with probability one,
either $D(t)>0$ eventually or $D(t)<0$ eventually.

Consider the event on which $D(t)>0$ eventually. Then there exists a
finite random time $T_0$ such that
\[
D(t)>0,\qquad t\geq T_0,
\]
so agent $1$ is the payoff leader thereafter. Moreover,
\[
\liminf_{t\to\infty}\frac{D(t)}{\sqrt{t}}
\geq 2a
\qquad\text{a.s.}
\]
Choose a deterministic constant $c\in(0,2a)$. Then, for all sufficiently
large $t$,
\[
D(t)\geq c\sqrt{t}.
\]

Let
\[
0<T_1<T_2<\cdots
\]
be the successive interaction times of the unique edge. The
interarrival times are i.i.d. $\operatorname{Exp}(1)$ random variables.
Hence, by the strong law of large numbers,
\[
\frac{T_k}{k}\longrightarrow1
\quad\text{as}\quad k\to\infty\quad\text{a.s.}
\]
In particular, $T_k\geq k/2$ for all sufficiently large $k$. Therefore,
for all sufficiently large $k$,
\[
D(T_k)\geq c\sqrt{T_k}
\geq c\sqrt{\frac{k}{2}}.
\]

At time $T_k$, the conditional probability that agent $1$ loses the
interaction is
\[
q_k
=
\frac{1}{1+e^{\beta D(T_k)}}
\leq e^{-\beta D(T_k)}
\leq
\exp\left(-\beta c\sqrt{\frac{k}{2}}\right).
\]
Thus
\[
\sum_{k=1}^{\infty}q_k<\infty
\qquad\text{a.s.}
\]
By the conditional Borel--Cantelli lemma, agent $1$ loses only finitely
many interactions almost surely. Hence there exists a finite random
index $K$ such that agent $1$ wins every interaction after $T_K$.

Suppose first that $0<\alpha<1$. After $T_K$, every interaction
multiplies agent $1$'s garbage by $1-\alpha$. Since the rate-one Poisson
process has infinitely many interaction times almost surely,
\[
x_1(t)
=
(1-\alpha)^{N(t)-N(T_K)}x_1(T_K)
\longrightarrow0
\quad\text{as}\quad t\to\infty\quad\text{a.s.},
\]
where $N(t)$ denotes the number of interactions up to time $t$.
By conservation of total garbage,
\[
x_2(t)=M-x_1(t)
\longrightarrow M
\quad\text{as}\quad t\to\infty\quad\text{a.s.}
\]

If $\alpha=1$, agent $1$'s garbage becomes exactly zero at the first
interaction after $T_K$. Hence
\[
x_1(t)\longrightarrow0,
\qquad
x_2(t)\longrightarrow M
\quad\text{as}\quad t\to\infty\quad\text{a.s.}
\]

The case in which $D(t)<0$ eventually is symmetric and yields
\[
x_2(t)\longrightarrow0,
\qquad
x_1(t)\longrightarrow M
\quad\text{as}\quad t\to\infty\quad\text{a.s.}
\]
Therefore, there exists a random variable $I\in[2]$ such that
\[
x_I(t)\longrightarrow0,
\qquad
x_{3-I}(t)\longrightarrow M
\quad\text{as}\quad t\to\infty\quad\text{a.s.}
\]
\end{proof}

\subsection{Fermi Dynamics on connected Graphs}

We now consider the garbage redistribution process under complete disposal. By tracking the number of agents holding positive garbage, we show that this number is nonincreasing. The connectivity of the graph and the positivity of the Fermi interaction probabilities are then used to establish finite-time garbage concentration.

\begin{proof}[\bf Proof of Theorem~\ref{Theorem:Garbage-concentration-alpha-one}]
Define
\[
S(t)=\#\{i\in[n]:x_i(t)>0\}.
\]
Since $\alpha=1$, whenever an edge with two positive-garbage
endpoints interacts, the winner's garbage is completely transferred to
the loser. Hence
\[
(+,+)\longrightarrow(0,+),
\]
and therefore $S(t)$ decreases by one. In particular, $S(t)$ is a
nonincreasing, integer-valued process. Thus there exist a finite random
time $T$ and a random variable $S\in\{1,\ldots,n\}$ such that
\[
S(t)=S
\qquad\text{for all }t\geq T
\quad\text{a.s.}
\]

We claim that $S=1$ almost surely. Suppose, to the contrary, that
\[
\mathbb P(S\geq2)>0.
\]
Consider a realization on the event $\{S\geq2\}$. For all $t\geq T$,
there are at least two agents with positive garbage. Choose two such
agents, say $v_0$ and $v_r$. Since $G$ is finite and connected, there is
a finite path
\[
v_0,v_1,\ldots,v_r
\]
connecting them.

We show that, conditional on the state at time $T$, there is a
positive probability that an edge with two positive-garbage
endpoints is created after time $T$. Consider the following finite
sequence of events. First, the edge $\{v_0,v_1\}$ interacts and $v_0$
wins. Then $\{v_1,v_2\}$ interacts and $v_1$ wins, and so on, until
$\{v_{r-1},v_r\}$ interacts and $v_{r-1}$ wins.

After the first interaction, $v_1$ has positive garbage. After the
second, $v_2$ has positive garbage, and this continues along the path.
Consequently, immediately after the last prescribed interaction,
both $v_{r-1}$ and $v_r$ have positive garbage. Thus
$\{v_{r-1},v_r\}$ is a $++$ edge.

Every prescribed edge interaction has positive probability of
being the next interaction, because each edge has a rate-one Poisson
clock. Moreover, conditional on the state immediately before each
prescribed interaction, the required winner has positive
Fermi probability, since
\[
0<
\frac{1}{1+e^{-\beta(P_i-P_j)}}<1.
\]
Hence the entire finite sequence has positive conditional
probability.

Once the resulting $++$ edge interacts, its state changes from
\[
(+,+)\longrightarrow(0,+),
\]
so that
\[
S(t+)=S(t)-1=S-1
\]
at some finite time after $T$. This contradicts
\[
S(t)=S
\qquad\text{for all }t\geq T.
\]
Therefore
\[
\mathbb P(S\geq2)=0,
\]
and hence
\[
S=1
\qquad\text{a.s.}
\]

Since the total garbage is conserved and equals $M$, all garbage is
concentrated on a single agent at each time after some finite random
time. Therefore there exists a finite random time $T<\infty$ such that
\[
x_{(n)}(t)=M
\qquad\text{for all }t\geq T
\quad\text{a.s.}
\]
\end{proof}

\section{Simulations}\label{sec:simulations}

In this section, we present numerical simulations to illustrate
the theoretical results established in Section~\ref{sec:main}.
We consider two settings: a competitive garbage-offloading
process on a random connected graph and the two-agent model
with different disposal efficiencies.

All simulations are performed using the Gillespie algorithm.
Since each edge is equipped with an independent rate-one
Poisson clock, the waiting times between consecutive
interactions are exponentially distributed with parameter
$|E|$. At each interaction, an edge is selected uniformly
from $E$, and the winner is determined according to the
Fermi rule. Throughout the simulations, we set
$M=1$, $a=0.01$, and $\beta=1$, with initially equal
garbage loads and zero payoffs.

\subsection{Payoff dispersion and garbage concentration
on a connected graph}

We first consider a random connected graph with $n=100$ vertices, generated by constructing a random spanning tree and subsequently adding each remaining edge independently with probability $0.01$.

We set $\alpha=1$ and simulate the process over the time
interval $[0,150]$. Figure~\ref{fig:general} displays
the evolution of four quantities associated with the
payoff and garbage distributions.

Figure~\ref{fig:general}(a) shows the scaled payoff
dispersion $V(t)/t$, together with the theoretical lower
bound $2a^2|E|$. Theorem~\ref{Thm:Almost-sure linear growth of payoff dispersion} establishes
that
\[
\liminf_{t\to\infty}\frac{V(t)}{t}
\geq 2a^2|E|
\qquad \text{a.s.}
\]
The simulated trajectory illustrates the growth of payoff
dispersion induced by the competitive interactions.

Figure~\ref{fig:general}(b) displays the maximum and
minimum payoffs. Their separation illustrates the
polarization of payoffs predicted by the theoretical
results. In particular, the maximum payoff tends to
increase while the minimum payoff tends to decrease,
reflecting the emergence of payoff inequality among
agents.

Figure~\ref{fig:general}(c) presents the scaled payoff
range $R(t)/\sqrt{t}$, together with the theoretical
lower bound
\[
2a\sqrt{\frac{|E|}{n-1}}.
\]
Corollary~\ref{Corollary:Almost-sure growth of the payoff range} establishes that
\[
\liminf_{t\to\infty}\frac{R(t)}{\sqrt{t}}
\geq 2a\sqrt{\frac{|E|}{n-1}}
\qquad \text{a.s.}
\]
This result provides a quantitative description of the
growth of payoff inequality.

Finally, Figure~\ref{fig:general}(d) shows the maximum
and minimum garbage loads. Since $\alpha=1$, the
theoretical results guarantee that, on a connected graph,
all garbage is eventually concentrated on a single
agent at each sufficiently large time. In terms of the
garbage order statistics, this means that
\[
x_{(n)}(t)=M
\]
for all sufficiently large $t$, almost surely.
The numerical experiment illustrates the evolution
of the garbage distribution toward this concentrated
configuration.

\begin{figure}[H]
\centering

\begin{minipage}{0.49\textwidth}
\centering
\includegraphics[width=\linewidth]{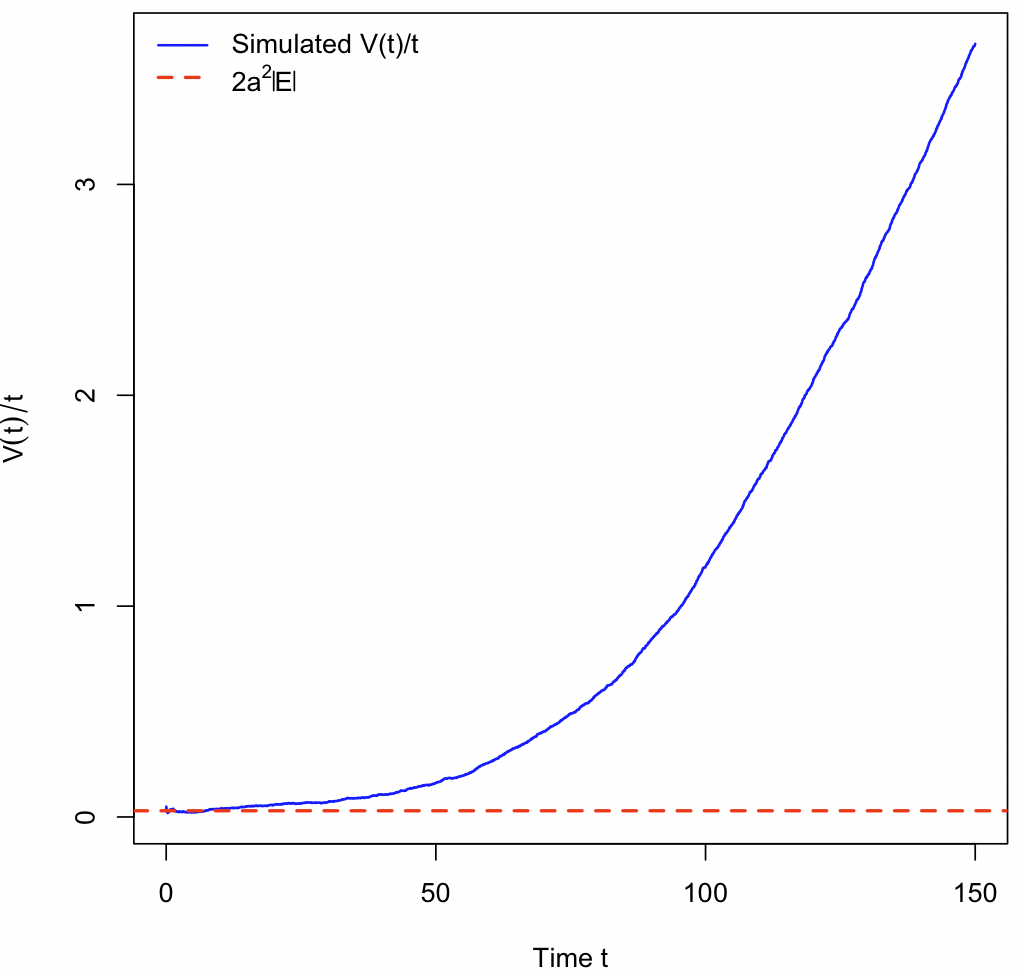}
\subcaption{Payoff dispersion}
\end{minipage}
\hfill
\begin{minipage}{0.49\textwidth}
\centering
\includegraphics[width=\linewidth]{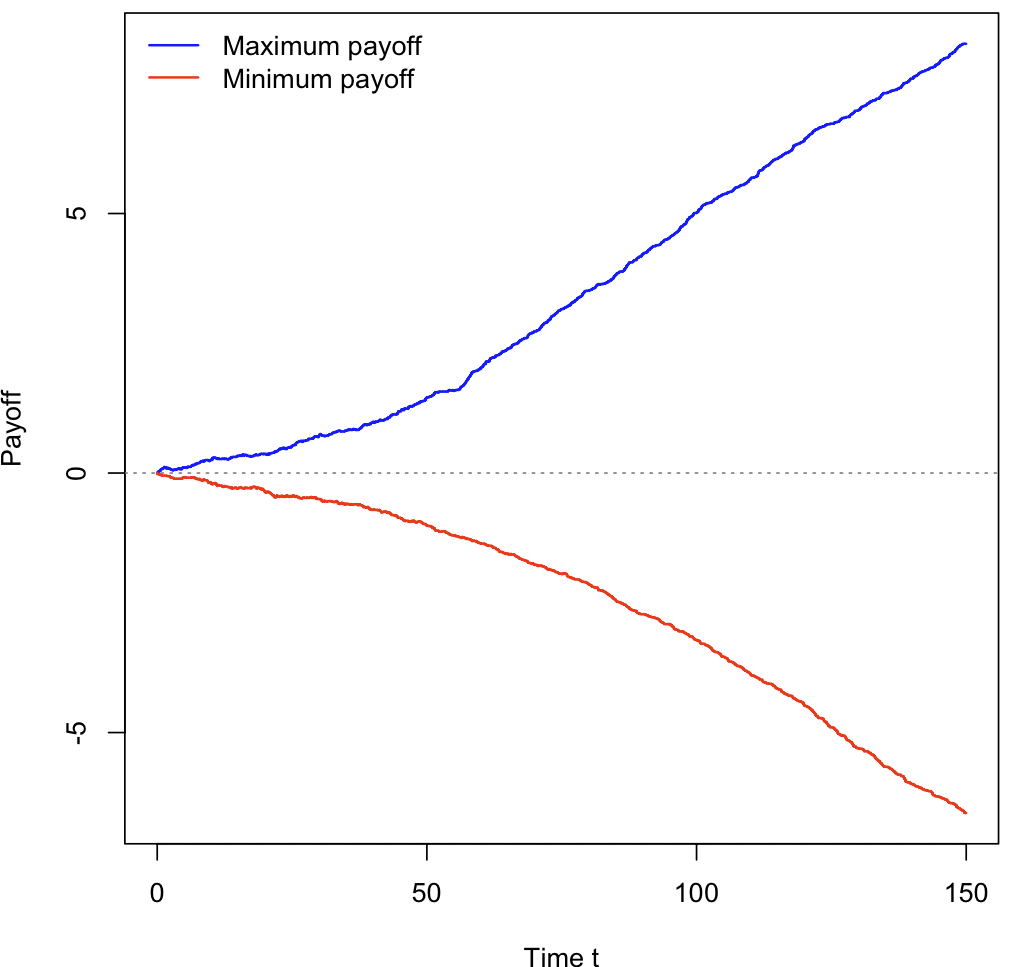}
\subcaption{Payoff polarization}
\end{minipage}

\vspace{0.3cm}

\begin{minipage}{0.49\textwidth}
\centering
\includegraphics[width=\linewidth]{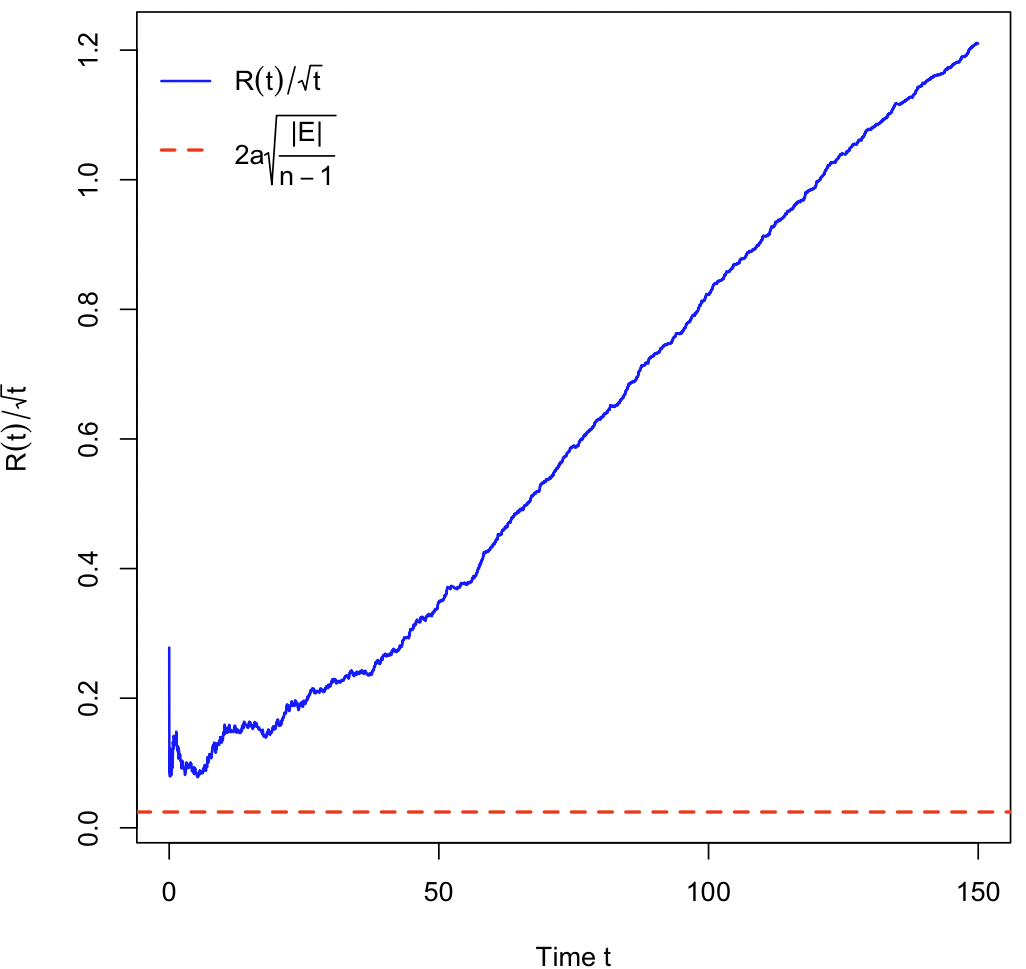}
\subcaption{Payoff range}
\end{minipage}
\hfill
\begin{minipage}{0.49\textwidth}
\centering
\includegraphics[width=\linewidth]{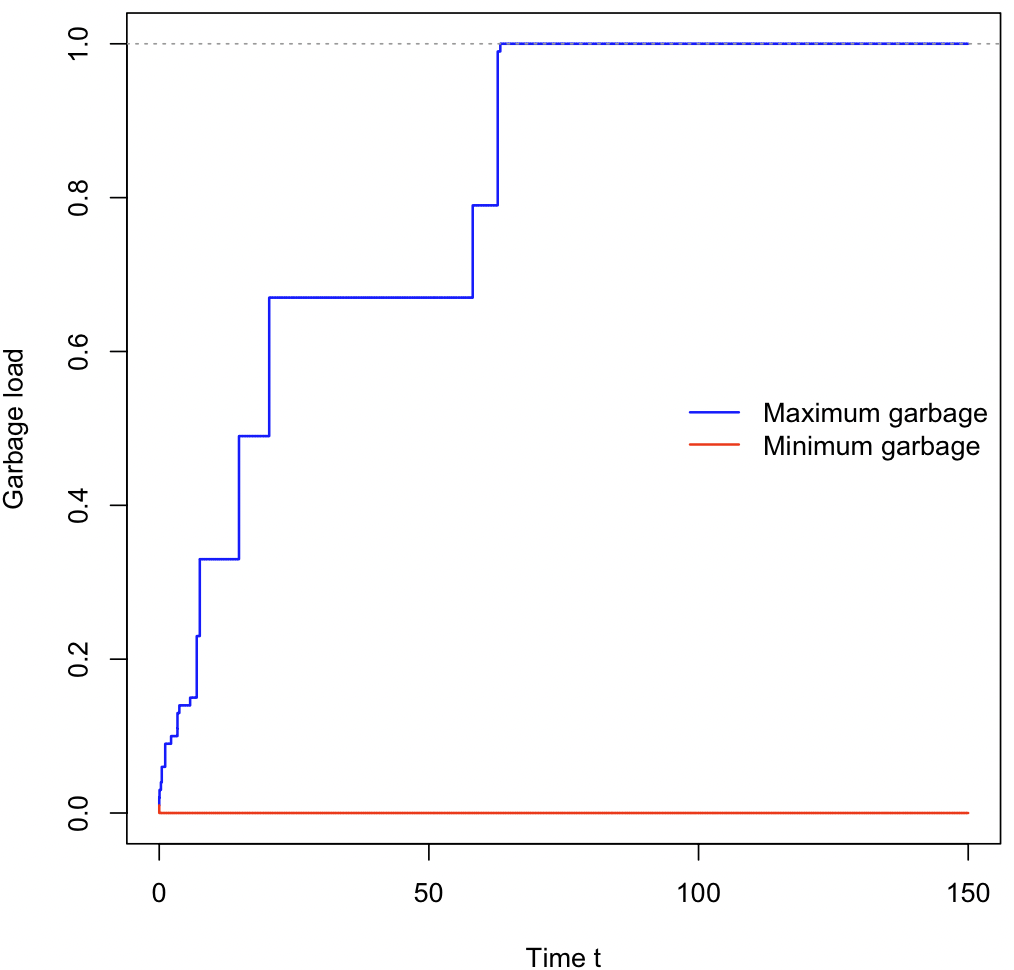}
\subcaption{Maximum and minimum garbage loads}
\end{minipage}

\caption{Evolution of the competitive garbage-offloading
process on a random connected graph with $n=100$,
$\alpha=1$, $\beta=1$, $a=0.01$, and $M=1$.
(a) Scaled payoff dispersion $V(t)/t$ and its
theoretical asymptotic lower bound.
(b) Maximum and minimum payoffs.
(c) Scaled payoff range $R(t)/\sqrt{t}$ and its
theoretical asymptotic lower bound.
(d) Maximum and minimum garbage loads.}
\label{fig:general}
\end{figure}

\subsection{Garbage concentration in the two-agent model}

We next consider the two-agent model, corresponding to $n=2$ and $E=\{\{1,2\}\}$. Unlike the general concentration result expressed in terms of garbage order statistics, Corollary~\ref{Corollary:Two-agent-garbage-concentration} establishes convergence of the garbage loads of individual agents. To illustrate this result and investigate the effect of disposal efficiency, we consider
\[
\alpha\in\{0.25,0.5,0.75,1\}.
\]
For each value of $\alpha$, we set $\beta=1$, $a=0.01$, and $M=1$, and simulate the process over the time interval $[0,1000]$.

Figure~\ref{fig:twoagents} displays the garbage trajectories $x_1(t)$ and $x_2(t)$ for the four values of $\alpha$.

By Corollary~\ref{Corollary:Two-agent-garbage-concentration}, for every $\alpha\in(0,1]$, there exists a random index $I\in\{1,2\}$ such that
\[
x_I(t)\longrightarrow 0,
\qquad
x_{3-I}(t)\longrightarrow M
\qquad\text{as }t\to\infty\quad\text{a.s.}
\]
Thus, all the garbage asymptotically concentrates on one fixed agent.

The simulations illustrate the influence of $\alpha$ on the transient behavior of the process. When $\alpha<1$, each interaction transfers only a fraction of the winner's garbage, and the garbage loads may fluctuate before approaching a concentrated configuration. By contrast, when $\alpha=1$, the winner transfers its entire garbage load to the loser at each interaction. Consequently, the garbage distribution becomes concentrated on one agent after the first interaction and remains concentrated thereafter.

These experiments illustrate the distinction between concentration in terms of garbage order statistics and convergence of individual garbage loads. In the two-agent model, the latter property holds for every $\alpha\in(0,1]$, although the disposal efficiency influences the transient dynamics.

\begin{figure}[H]
\centering

\begin{minipage}{0.49\textwidth}
\centering
\includegraphics[width=\linewidth]{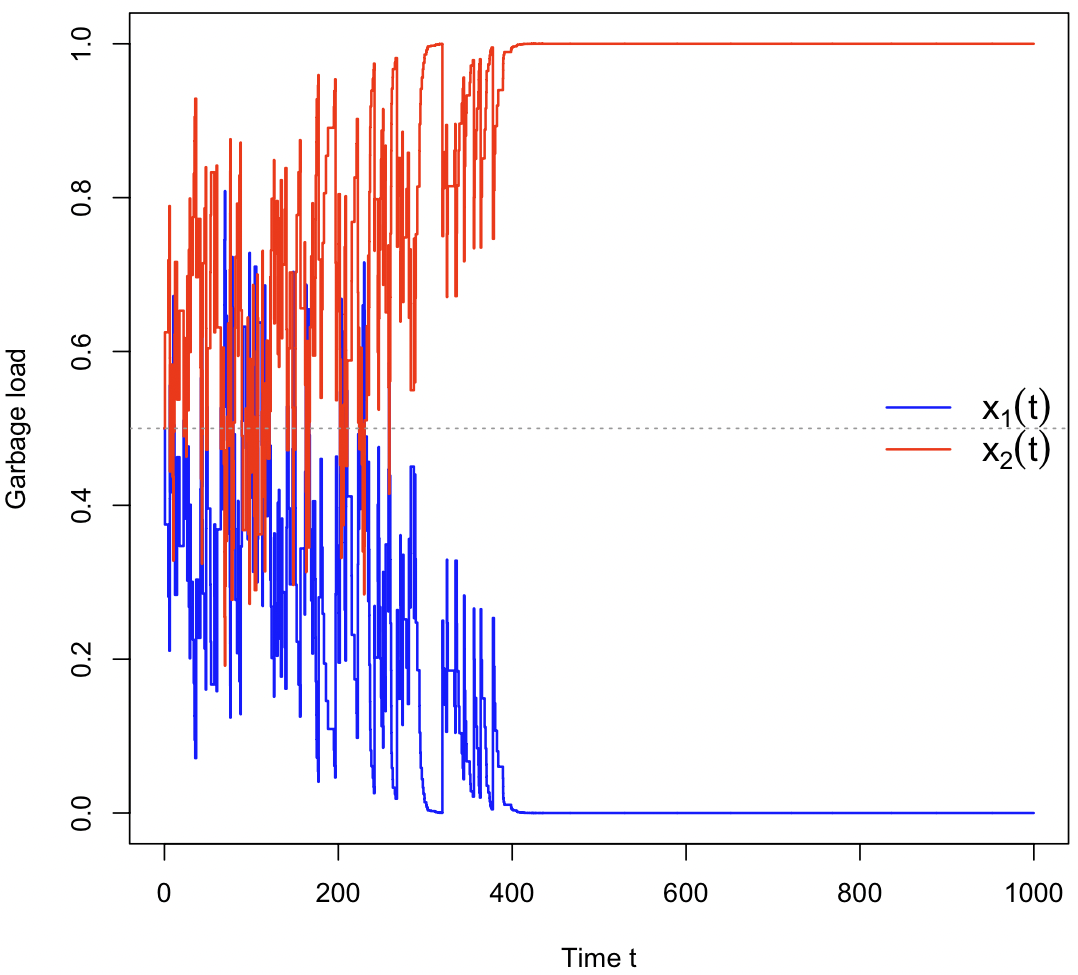}
\subcaption{$\alpha=0.25$}
\end{minipage}
\hfill
\begin{minipage}{0.49\textwidth}
\centering
\includegraphics[width=\linewidth]{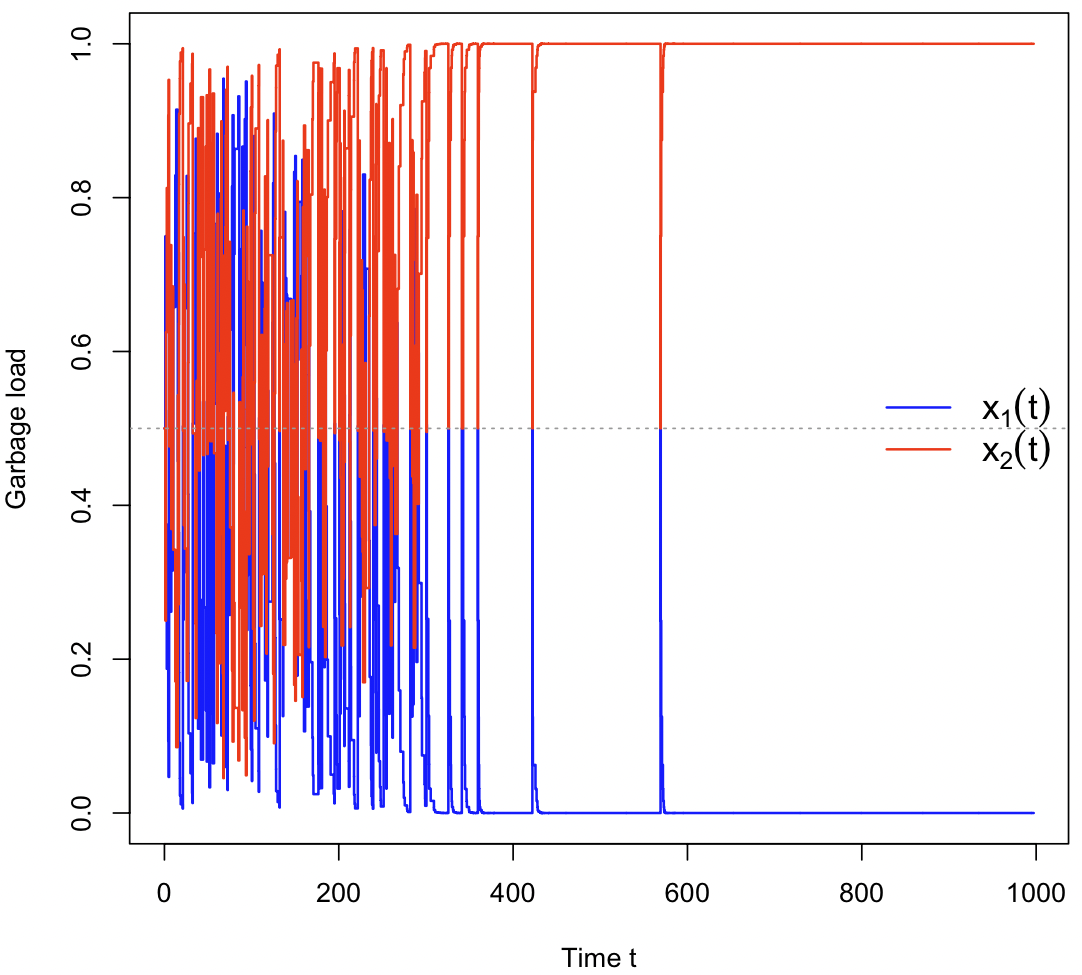}
\subcaption{$\alpha=0.5$}
\end{minipage}

\vspace{0.3cm}

\begin{minipage}{0.49\textwidth}
\centering
\includegraphics[width=\linewidth]{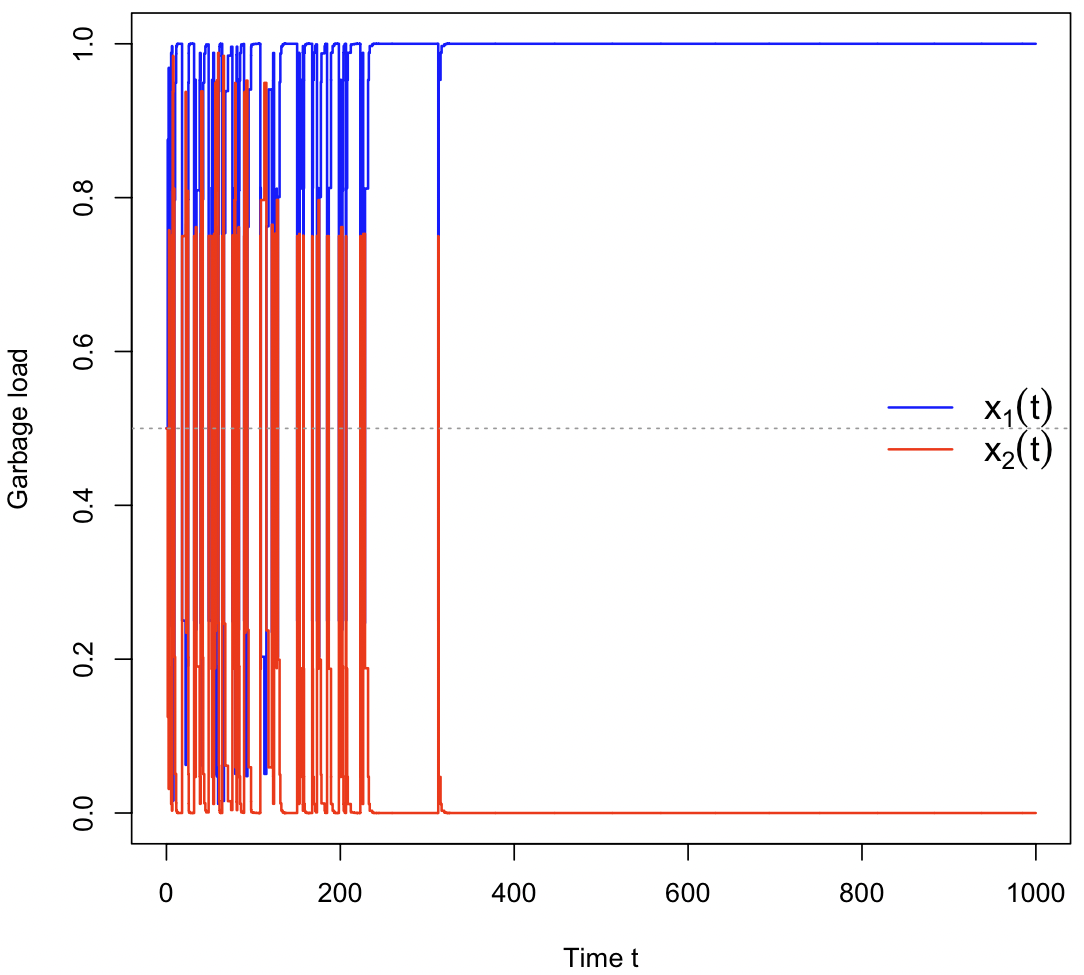}
\subcaption{$\alpha=0.75$}
\end{minipage}
\hfill
\begin{minipage}{0.49\textwidth}
\centering
\includegraphics[width=\linewidth]{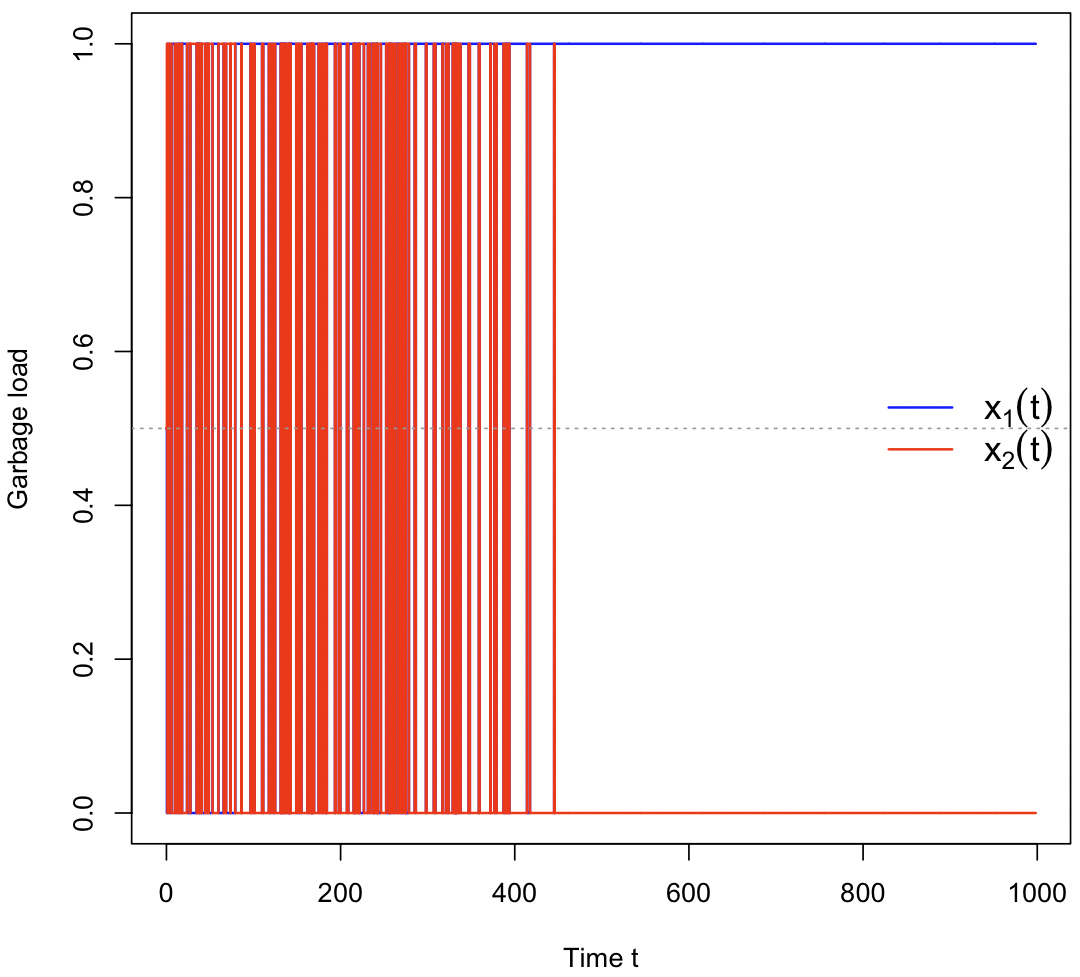}
\subcaption{$\alpha=1$}
\end{minipage}

\caption{Garbage trajectories in the two-agent model
with $\beta=1$, $a=0.01$, and $M=1$.
The blue and red curves represent $x_1(t)$ and
$x_2(t)$, respectively. The horizontal dotted line
indicates $M/2$.}
\label{fig:twoagents}
\end{figure}

\section{Discussion and Future Research}

In this paper, we investigated a competitive garbage-offloading
process on a finite graph, in which agents transfer garbage
according to the outcomes of pairwise interactions governed
by the Fermi function. The model incorporates both payoff
dynamics and garbage redistribution, allowing us to examine
the long-term behavior of these two interconnected processes.

Our results show that competitive interactions lead to
unbounded payoff inequality. In particular, the payoff
dispersion grows at least linearly in time in the asymptotic
sense, while the payoff range grows at least on the order
of $\sqrt{t}$. These conclusions hold for arbitrary finite
graphs with at least one edge and are independent of the
disposal efficiency. The numerical simulations illustrate
the theoretical lower bounds and the polarization of
payoffs over time.

For garbage redistribution, we established that when
$\alpha=1$ and the graph is connected, all garbage becomes
concentrated on a single agent after a finite random time,
almost surely. However, the identity of the agent holding
all the garbage may change over time. For the two-agent
model, we obtained a stronger conclusion: for every
$\alpha\in(0,1]$, all garbage asymptotically concentrates
on one fixed agent, almost surely. These results demonstrate
that garbage concentration can arise under competitive
interactions, although its precise form depends on the
disposal efficiency and the graph structure.

Several questions remain open. One natural direction is
to investigate garbage concentration for
$\alpha\in(0,1)$ on general connected graphs. In
particular, it would be interesting to determine whether
the maximum garbage load converges to the total garbage
and whether the garbage loads of individual agents
converge. Another direction is to study the influence
of graph topology on the rate of garbage concentration,
especially for paths, stars, and complete graphs.

Further research could also examine the quantitative
behavior of payoff dispersion and payoff polarization.
Although the present results provide asymptotic lower
bounds, obtaining sharper growth estimates or identifying
limiting distributions would offer a more complete
description of the payoff dynamics. Finally, extending
the model to heterogeneous interaction rates or
time-dependent networks may provide additional insight
into the relationship between competitive interactions,
payoff inequality, and garbage redistribution.

\section{Statements and Declarations}

\subsection{Competing Interests}
The author is partially funded by an NSTC grant.

\subsection{Data Availability}
No associated data were used in this study.

\end{document}